\documentclass{amsart}
\usepackage{amsmath,amssymb,amsthm,bm,bbm}
\usepackage{amscd}
\usepackage{graphicx,enumerate}
\usepackage{layout}
\usepackage{longtable}
\usepackage{comment}
\usepackage[OT2,T1]{fontenc}
\DeclareSymbolFont{cyrletters}{OT2}{wncyr}{m}{n}
\DeclareMathSymbol{\Sha}{\mathalpha}{cyrletters}{"58}
\usepackage[OT2,OT1]{fontenc}
\newcommand\cyr{\renewcommand\rmdefault{wncyr}
\renewcommand\sfdefault{wncyss}
\renewcommand\encodingdefault{OT2}
\normalfont\selectfont}
\DeclareTextFontCommand{\textcyr}{\cyr}

\usepackage{url}

\theoremstyle{plain}
\newtheorem{theorem}{Theorem}[section]
\newtheorem*{theorem-nn}{Theorem}

\newtheorem{proposition}[theorem]{Proposition}
\newtheorem*{proposition-nn}{Proposition}
\newtheorem{corollary}[theorem]{Corollary}

\theoremstyle{definition}
\newtheorem{definition}[theorem]{Definition}

\newtheorem{remark}[theorem]{Remark}

\theoremstyle{remark}

\newcommand{\bZ}{\mathbbm{Z}}\newcommand{\bQ}{\mathbbm{Q}}
\newcommand{\bC}{\mathbbm{C}}\newcommand{\bG}{\mathbbm{G}}
\newcommand{\bF}{\mathbbm{F}}

\newcommand{\bA}{\mathbbm{A}}
\newcommand{\GL}{{\rm GL}}\newcommand{\SL}{{\rm SL}}
\newcommand{\PSL}{{\rm PSL}}
\newcounter{sub}
{\begin{list}{(\arabic{sub})}{\usecounter{sub}%
\setlength{\leftmargin}{2em}}}{\end{list}}

\title{Hasse norm principle for $M_{11}$ and $J_1$ extensions}

\author[A. Hoshi]{Akinari Hoshi}
\address{Department of Mathematics, Niigata University, Niigata 950-2181, Japan}
\email{hoshi@math.sc.niigata-u.ac.jp}

\author[K. Kanai]{Kazuki Kanai}
\address{General Education Program, National Institute of Technology, Kure College, Hiroshima 737-8506, Japan}
\email{k-kanai@kure-nct.ac.jp}

\author[A. Yamasaki]{Aiichi Yamasaki}
\address{Department of Mathematics, Kyoto University, Kyoto 606-8502, Japan}
\email{aiichi.yamasaki@gmail.com}

\thanks{
This work was partially supported by JSPS KAKENHI Grant Numbers 
19K03418, 20H00115, 20K03511. 
}

\begin{document}
\begin{abstract}
We give a necessary and sufficient condition for the Hasse norm principle 
for field extensions $K/k$ when the Galois groups ${\rm Gal}(L/k)$ 
of the Galois closure $L/k$ of $K/k$  are isomorphic to 
the Mathieu group $M_{11}$ of degree $11$ of order $7920$ 
or the Janko group $J_1$ of order $175560$ 
both with trivial Schur multiplier 
by determining $H^1(k,{\rm Pic}\, \overline{X})=0$ or $\bZ/2\bZ$  
for norm one tori $T=R^{(1)}_{K/k}(\bG_m)$ with a smooth 
$k$-compactification $X$ and $\overline{X}=X\times_k\overline{k}$. 
The result gives a first step towards understanding the all pictures of 
the Hasse norm principle for the $26$ sporadic simple groups. 
\end{abstract}

\maketitle

\tableofcontents

%%%%%%%%%%%%%%%%%%%%%%%%%%%%%%%%%%%%%%%%%%%%%%%%%%%%%%%%
\section{Introduction}\label{seInt}

Let $k$ be a field, 
$\overline{k}$ be a fixed separable closure of $k$ and 
$\mathcal{G}={\rm Gal}(\overline{k}/k)$ be the absolute Galois group of $k$. 
%%%%%%%%%%%%%%%%%%%%%%%%%%%%%%%%%%%%%%%%
%
Let $T$ be an algebraic $k$-torus, 
i.e. a group $k$-scheme with fiber product (base change) 
$T\times_k \overline{k}=
T\times_{{\rm Spec}\, k}\,{\rm Spec}\, \overline{k}
\simeq (\bG_{m,\overline{k}})^n$; 
$k$-form of the split torus $(\bG_m)^n$. 
Then there exists the minimal (canonical) finite Galois extension $K/k$ 
with Galois group $G={\rm Gal}(K/k)$ such that 
$T$ splits over $K$: $T\times_k K\simeq (\bG_{m,K})^n$ 
(see Voskresenskii \cite[page 27, Example 6]{Vos98}). 

Let $k$ be a global field, 
i.e. a number field (a finite extension of $\bQ$) 
or a function field of an algebraic curve over 
$\bF_q$ (a finite extension of $\bF_q(t))$. 
Let $T(k)$ be the group of $k$-rational points of $T$. 
Then $T(k)$ 
embeds into $\prod_{v\in V_k} T(k_v)$ by the diagonal map 
where %$T(k_v)$ is the group of $k_v$-rational points and 
$V_k$ is the set of all places of $k$ and 
$k_v$ is the completion of $k$ at $v\in V_k$. 
Let $\overline{T(k)}$ be the closure of $T(k)$  
in the product $\prod_{v\in V_k} T(k_v)$. 
The group 
\begin{align*}
A(T)=\left(\prod_{v\in V_k} T(k_v)\right)/\overline{T(k)}
\end{align*}
is called {\it the kernel of the weak approximation} of $T$. 
We say that {\it $T$ has the weak approximation property} if $A(T)=0$. 

Let $E$ be a principal homogeneous space (= torsor) under $T$.  
{\it Hasse principle holds for $E$} means that 
if $E$ has a $k_v$-rational point for all $k_v$, 
then $E$ has a $k$-rational point. 
The set $H^1(k,T)$ classifies all such torsors $E$ up 
to (non-unique) isomorphism. 
We define {\it the Shafarevich-Tate group}
\begin{align*}
\Sha(T)={\rm Ker}\left\{H^1(k,T)\xrightarrow{\rm res} \bigoplus_{v\in V_k} 
H^1(k_v,T)\right\}.
\end{align*}
%We say that $T$ {\it satisfies the Hasse principle} 
%if $\Sha(T)=0$, i.e. Hasse principle holds for all torsors $E$ under $T$. 
Then 
Hasse principle holds for all torsors $E$ under $T$ 
if and only if $\Sha(T)=0$. 
\begin{theorem}[{Voskresenskii \cite[Theorem 5, page 1213]{Vos69}, 
\cite[Theorem 6, page 9]{Vos70}, see also \cite[Section 11.6, Theorem, page 120]{Vos98}}]\label{thV}
Let $k$ be a global field, 
$T$ be an algebraic $k$-torus, 
$X$ be a smooth $k$-compactification of $T$ and 
${\rm Pic}\,\overline{X}$ be the Picard group of $\overline{X}=X\times_k\overline{k}$.  
Then there exists an exact sequence
\begin{align*}
0\to A(T)\to H^1(k,{\rm Pic}\,\overline{X})^{\vee}\to \Sha(T)\to 0
\end{align*}
where $M^{\vee}={\rm Hom}(M,\bQ/\bZ)$ is the Pontryagin dual of $M$. 
Moreover, if $L$ is the splitting field of $T$ and $L/k$ 
is an unramified extension, then $A(T)=0$ and 
$H^1(k,{\rm Pic}\,\overline{X})^{\vee}\simeq \Sha(T)$. 
\end{theorem}
For the last assertion, see \cite[Theorem, page 120]{Vos98}.
It follows that 
$H^1(k,{\rm Pic}\,\overline{X})=0$ if and only if $A(T)=0$ and $\Sha(T)=0$, 
i.e. $T$ has the weak approximation property and 
Hasse principle holds for all torsors $E$ under $T$. 
Theorem \ref{thV} was generalized 
to the case of linear algebraic groups by Sansuc \cite{San81}.
%%%%%%%%%%%%%%%%%%%%%%%%%%%%%%%%%%%%%%%%%%%%%%%%%%%%%

Let $G$ be a finite group and $M$ be a $G$-lattice, 
i.e. finitely generated $\bZ[G]$-module which is $\bZ$-free as an abelian group.
We define 
\begin{align*}
\Sha^i_\omega(G,M):={\rm Ker}\left\{H^i(G,M)\xrightarrow{{\rm res}}\bigoplus_{g\in G}H^i(\langle g\rangle,M)\right\}\quad (i\geq 1) .
\end{align*}
The following is a theorem of Colliot-Th\'{e}l\`{e}ne and Sansuc \cite{CTS87}: 
\begin{theorem}[{Colliot-Th\'{e}l\`{e}ne and Sansuc \cite[Proposition 9.5 (ii)]{CTS87}, see also \cite[Proposition 9.8]{San81} and \cite[page 98]{Vos98}}]
Let $k$ be a field 
with ${\rm char}\, k=0$
and $K/k$ be a finite Galois extension 
with Galois group $G={\rm Gal}(K/k)$. 
Let $T$ be an algebraic $k$-torus which splits over $K$ with the character lattice 
$\widehat{T}={\rm Hom}(T,\bG_m)$ 
and 
$X$ be a smooth $k$-compactification of $T$. 
Then we have 
\begin{align*}
\Sha^2_\omega(G,\widehat{T})\simeq 
H^1(G,{\rm Pic}\, X_K)\simeq {\rm Br}(X)/{\rm Br}(k)
\end{align*}
where 
${\rm Br}(X)$ is the \'etale cohomological Brauer Group of $X$ 
$($it is the same as the Azumaya-Brauer group of $X$ 
for such $X$, see \cite[page 199]{CTS87}$)$. 
\end{theorem}

In other words, for $G$-lattice $M=\widehat{T}$, 
we have 
$H^1(k,{\rm Pic}\, \overline{X})\simeq H^1(G,{\rm Pic}\, X_K)\simeq 
%H^1(G,[M]^{fl})\simeq 
\Sha^2_\omega(G,M)\simeq {\rm Br}(X)/{\rm Br}(k)$ 
%(for the flabby class $[M]^{fl}$ of $M$, see Section \ref{S3}).
%Hence Theorem \ref{thKun1}, 
%Theorem \ref{thmain1-4} and Theorem \ref{thmain1-5} compute 
%$H^1(G,[M]^{fl})\simeq\Sha^2_\omega(G,M)\simeq {\rm Br}(X)/{\rm Br}(k)$ 
%where $M=\widehat{T}$. 
%%%%%
We also see  
%${\rm Br}_{\rm nr}(k(X)/k)\subset {\rm Br}\, X\subset {\rm Br}\, k(X)$ 
%and if $X$ is complete, then 
${\rm Br}_{\rm nr}(k(X)/k)={\rm Br}(X)\subset {\rm Br}(k(X))$ 
(see Colliot-Th\'{e}l\`{e}ne and Sansuc \cite[Theorem 5.11]{CTS07}, 
Saltman \cite[Proposition 10.5]{Sal99}). 
Note that if $H^1(k,{\rm Pic}\, \overline{X})\neq 0$, then $X$ (resp. $T$) is not retract $k$-rational 
(see e.g. Hoshi, Kanai and Yamasaki \cite[Section 3]{HKY22}, \cite[Section 4]{HKY23}).\\ 

Let $k$ be a field and $K/k$ be a finite extension. 
%%%%%%%%%%%%%%%%%%%%%%%%%%%%%%%%%%%%%%%%%%%%%%%%%%%%%
{\it The norm one torus 
$R^{(1)}_{K/k}(\bG_m)$ of $K/k$} 
is the kernel of the norm map 
$R_{K/k}(\bG_m)\rightarrow \bG_m$ 
where 
$R_{K/k}$ is the Weil restriction 
(see \cite[page 37, Section 3.12]{Vos98}).
Such a torus $T=R^{(1)}_{K/k}(\bG_m)$ (resp. a torsor under $T=R^{(1)}_{K/k}(\bG_m)$) 
is biregularly isomorphic to the norm hypersurface 
$f(x_1,\ldots,x_n)=1$ (resp. $f(x_1,\ldots,x_n)=a$ for some $a\in k^\times$) 
where $f\in k[x_1,\ldots,x_n]$ is the polynomial of total 
degree $n$ defined by the norm map $N_{K/k}:K^\times\to k^\times$ 
(see \cite[page 53, Section 4.8, page 122, Example 4]{Vos98}). 
When $K/k$ is a finite Galois extension, %i.e. $H=1$, 
we also have: 

\begin{theorem}[{Voskresenskii \cite[Theorem 7]{Vos70}, Colliot-Th\'{e}l\`{e}ne and Sansuc \cite[Proposition 1]{CTS77}}]
Let $k$ be a field and 
$K/k$ be a finite Galois extension with Galois group $G={\rm Gal}(K/k)$. 
Let $T=R^{(1)}_{K/k}(\bG_m)$ be the norm one torus of $K/k$ 
and $X$ be a smooth $k$-compactification of $T$. 
Then 
$H^1(H,{\rm Pic}\, X_K)\simeq H^3(H,\bZ)$ for any subgroup $H$ of $G$. 
In particular, 
$H^1(k,{\rm Pic}\, \overline{X})\simeq
H^1(G,{\rm Pic}\, X_K)\simeq H^3(G,\bZ)$ which is isomorphic to 
the Schur multiplier $M(G)$ of $G$.
\end{theorem}

Let $k$ be a global field, 
$K/k$ be a finite extension and 
$\bA_K^\times$ be the idele group of $K$. 
We say that {\it the Hasse norm principle holds for $K/k$} 
if $(N_{K/k}(\bA_K^\times)\cap k^\times)/N_{K/k}(K^\times)=1$ 
where $N_{K/k}$ is the norm map. 
Ono \cite{Ono63} established the relationship 
between the Hasse norm principle for $K/k$ 
and the Hasse principle for all torsors under 
the norm one torus $T=R^{(1)}_{K/k}(\bG_m)$ of $K/k$: 
\begin{theorem}[{Ono \cite[page 70]{Ono63}, see also Platonov \cite[page 44]{Pla82}, Kunyavskii \cite[Remark 3]{Kun84}, Platonov and Rapinchuk \cite[page 307]{PR94}}]\label{thOno}
Let $k$ be a global field and $K/k$ be a finite extension. 
Let $T=R^{(1)}_{K/k}(\bG_m)$ be the norm one torus of $K/k$.  
Then 
\begin{align*}
\Sha(T)\simeq (N_{K/k}(\bA_K^\times)\cap k^\times)/N_{K/k}(K^\times).
\end{align*}
In particular, $\Sha(T)=0$ if and only if 
the Hasse norm principle holds for $K/k$. 
\end{theorem}

Hasse \cite[Satz, page 64]{Has31} originally proved that 
the Hasse norm principle holds for any cyclic extension $K/k$ 
but does not hold for bicyclic extension $\bQ(\sqrt{-39},\sqrt{-3})/\bQ$. 
For finite Galois extensions $K/k$, Tate \cite{Tat67} gave the following theorem:

%%%%%%%%%%%%%

\begin{theorem}[{Tate \cite[page 198]{Tat67}}]\label{thTate}
Let $k$ be a global field, $K/k$ be a finite Galois extension 
with Galois group $G={\rm Gal}(K/k)$. 
Let $V_k$ be the set of all places of $k$ 
and $G_v$ be the decomposition group of $G$ at $v\in V_k$. 
Then 
\begin{align*}
(N_{K/k}(\bA_K^\times)\cap k^\times)/N_{K/k}(K^\times)\simeq 
{\rm Coker}\left\{\bigoplus_{v\in V_k}\widehat H^{-3}(G_v,\bZ)\xrightarrow{\rm cores}\widehat H^{-3}(G,\bZ)\right\}
\end{align*}
where $\widehat H$ is the Tate cohomology. 
In particular, the Hasse norm principle holds for $K/k$ 
if and only if the restriction map 
$H^3(G,\bZ)\xrightarrow{\rm res}\bigoplus_{v\in V_k}H^3(G_v,\bZ)$ 
is injective. 
\end{theorem}
If $G\simeq C_n$ is cyclic, then 
$\widehat H^{-3}(G,\bZ)\simeq H^3(G,\bZ)\simeq H^1(G,\bZ)={\rm Hom}(G,\bZ)=0$ 
and hence the Hasse's original theorem follows. 

The Hasse norm principle for Galois extensions $K/k$ 
was investigated by Gerth \cite{Ger77}, \cite{Ger78} and 
Gurak \cite{Gur78a}, \cite{Gur78b}, \cite{Gur80} 
(see also \cite[pages 308--309]{PR94}). 
For non-Galois extension $K/k$, 
the Hasse norm principle was investigated by 
Bartels \cite{Bar81a} ($[K:k]=p$; prime), 
\cite{Bar81b} (${\rm Gal}(L/k)\simeq D_n$), 
Voskresenskii and Kunyavskii \cite{VK84} (${\rm Gal}(L/k)\simeq S_n$) and 
Macedo \cite{Mac20} (${\rm Gal}(L/k)\simeq A_n$) 
where $L/k$ be the Galois closure of $K/k$, 
Macedo and Newton \cite{MN22}, 
Hoshi, Kanai and Yamasaki \cite{HKY22}, \cite{HKY23}.\\

%%%%%%%%%%%%%%%%%%%%%%%%%%%%%%%%%%%%%%%%%%

Let $M_{11}$ be the Mathieu group of degree $11$ 
and $J_1$ be the Janko group 
(see Dixon and Mortimer \cite[Chapter 6]{DM96}, 
Gorenstein, Lyons and Solomon \cite[Chapter 5]{GLS98}). 
Note that the groups $M_{11}$ and $J_1$ are two of 
the $26$ sporadic simple groups, and 
$|M_{11}|=7920=2^4\cdot 3^2\cdot 5\cdot 11$ and 
$|J_1|=175560=2^3\cdot 3\cdot 5\cdot 7\cdot 11\cdot 19$. 
For $G\simeq M_{11}$ and $J_1$, 
we have the trivial Schur multiplier 
$H^2(G,\bC^\times)\simeq H^3(G,\bZ)=0$ although 
$H^2(M_{12},\bC^\times)\simeq \bZ/2\bZ$ and 
$H^2(M_{22},\bC^\times)\simeq \bZ/12\bZ$. 
This is one of the reasons why we can reach the answer of 
the problem (see Section \ref{S3}). 
We determine $H^1(k,{\rm Pic}\, \overline{X})$ 
for norm one tori $T=R^{(1)}_{K/k}(\bG_m)$, 
when the Galois group of the Galois closure $L/k$ of $K/k$  
is isomorphic to the Mathieu group $M_{11}$ or the Janko group $J_1$, 
which depends on $G={\rm Gal}(L/k)\simeq M_{11}$ or $J_1$, 
and $H={\rm Gal}(L/K)\leq G$ up to conjugacy.  
Let ${\rm Syl}_2(G)$ be a $2$-Sylow subgroup of $G$. 

%%%%%%%%%%%%%%%%%%%%%%%%%%%%%%%%%%%%%%%%%%%%%%%%
%
\begin{theorem}\label{thmain1}
Let $k$ be a field, 
$K/k$ be a separable field extension of degree $n$ 
and $L/k$ be the Galois closure of $K/k$. 
Assume that $G={\rm Gal}(L/k)\simeq M_{11}$ and $H={\rm Gal}(L/K)\lneq G$ 
where $M_{11}$ is the Mathieu group of degree $11$. 
Let $T=R^{(1)}_{K/k}(\bG_m)$ be the norm one torus of $K/k$ of dimension $n-1$ 
and $X$ be a smooth $k$-compactification of $T$. 
Then we have 
\begin{align*}
H^1(k,{\rm Pic}\, \overline{X})=
\begin{cases}
0 & {\rm if}\ \ {\rm Syl}_2(H)\not\simeq C_2,C_4,C_8,\\
\bZ/2\bZ & {\rm if}\ \ {\rm Syl}_2(H)\simeq C_2,C_4,C_8 
\end{cases}
\end{align*}
as in Table $1$ and Table $2$.  
In particular, {\rm (i)} if  $H^1(k,{\rm Pic}\, \overline{X})\neq 0$, 
then $X$ (resp. $T$) is not retract $k$-rational; 
{\rm (ii)} if $k$ is a number field and 
$L/k$ is an unramified extension, then $A(T)=0$ and 
$H^1(k,{\rm Pic}\,\overline{X})\simeq \Sha(T)$. 
%In particular, if $k$ is a number field and $H^1(k,{\rm Pic}\, \overline{X})=0$, then 
%$A(T)=0$ and $\Sha(T)=0$. 
\end{theorem}
\begin{theorem}\label{thmain1-2}
Let $k$ be a field, 
$K/k$ be a separable field extension of degree $n$ 
and $L/k$ be the Galois closure of $K/k$. 
Assume that $G={\rm Gal}(L/k)\simeq J_1$ and $H={\rm Gal}(L/K)\lneq G$ 
where $J_1$ is the Janko group. 
Let $T=R^{(1)}_{K/k}(\bG_m)$ be the norm one torus of $K/k$ of dimension $n-1$ 
and $X$ be a smooth $k$-compactification of $T$. 
Then we have 
\begin{align*}
H^1(k,{\rm Pic}\, \overline{X})=
\begin{cases}
0 & {\rm if}\ \ {\rm Syl}_2(H)\not\simeq C_2,\\
\bZ/2\bZ & {\rm if}\ \ {\rm Syl}_2(H)\simeq C_2
\end{cases}
\end{align*}
as in Table $3$ and Table $4$.  
In particular, {\rm (i)} if  $H^1(k,{\rm Pic}\, \overline{X})\neq 0$, 
then $X$ (resp. $T$) is not retract $k$-rational; 
{\rm (ii)} if $k$ is a number field and 
$L/k$ is an unramified extension, then $A(T)=0$ and 
$H^1(k,{\rm Pic}\,\overline{X})\simeq \Sha(T)$.  
%In particular, if $k$ is a number field and $H^1(k,{\rm Pic}\, \overline{X})=0$, then 
%$A(T)=0$ and $\Sha(T)=0$. 
\end{theorem}
In Tables $1$--$4$, %Table $2$, Table $3$ and Table $4$, 
$C_n$ (resp. $D_n$, $QD_n$, $A_n$, $S_n$, $V_4$, $Q_8$) denotes 
the cyclic (resp. dihedral, quasi-dihedral, alternating, symmetric, Klein four, quaternion) 
group of order 
$n$ (resp. $2n$, $2n$, $n!/2$, $n!$, $4$, $8$)  
and 
$\SL_2(\bF_3)$ (resp. $\GL_2(\bF_3)$, $\PSL_2(\bF_{11})$) denotes 
the special (resp. general, projective special) linear group of degree $2$ over 
$\bF_3$ (resp. $\bF_3$, $\bF_{11}$) 
of order $24$ (resp. $48$, $660$). 
The subgroups $H^{(1)}$ and $H^{(2)}$ of $G$ are isomorphic to $H$ but not conjugate in $G$. 
See Section 4 %\cite[Section 3]{HY} 
for more detailed information and GAP computations. 

For the flabby class $[J_{G/H}]^{fl}$ of the Chevalley module 
$J_{G/H}=(I_{G/H})^\circ={\rm Hom}_\bZ(I_{G/H},\bZ)
\simeq \widehat{T}={\rm Hom}(T,\bG_m)$ where $I_{G/H}={\rm Ker}\,\varepsilon$ 
and $\varepsilon:\bZ[G/H]\to \bZ$ is the argumentation map, 
as in Table $1$ and Table $2$, 
see the related previous papers \cite{HY17}, \cite{HHY20}, \cite{HY21}, \cite{HKY22}, \cite{HKY23}.
For $G\simeq M_{11}$, 
it turns out that there exist $38=25+13$ subgroups $H\lneq G$ up to conjugacy 
and $25$ (resp. $13$) of them satisfiy 
$H^1(k,{\rm Pic}\, \overline{X})\simeq H^1(G,[J_{G/H}]^{fl})=0$ (resp. $\bZ/2\bZ$) 
(see Table $1$ and Table $2$). 
For $G\simeq J_1$, there exist $39=23+16$ subgroups $H\lneq G$ up to conjugacy 
and $23$ (resp. $16$) of them satisfy 
$H^1(k,{\rm Pic}\, \overline{X})\simeq H^1(G,[J_{G/H}]^{fl})=0$ (resp. $\bZ/2\bZ$) 
(see Table $3$ and Table $4$). 

Using Theorem \ref{thmain1} and Theorem \ref{thmain1-2}, 
we give a necessary and sufficient condition for the Hasse norm principle 
for $K/k$ (i.e. $\Sha(T)=0$, see Theorem \ref{thOno}) 
when the Galois closure $L/k$ of $K/k$ satisfies ${\rm Gal}(L/k)\simeq M_{11}$ or $J_1$. 
The following two results for $M_{11}$ and $J_1$ 
give a first step towards understanding the all pictures of 
the Hasse norm principle for the $26$ sporadic simple groups. 
\begin{theorem}\label{thmain2}
Let $k$ be a number field, %global field, %number field, 
$K/k$ be a field extension of degree $n$ 
and $L/k$ be the Galois closure of $K/k$. 
Assume that $G={\rm Gal}(L/k)\simeq M_{11}$ and $H={\rm Gal}(L/K)\lneq G$ 
where $M_{11}$ is the Mathieu group of degree $11$. 
Let $T=R^{(1)}_{K/k}(\bG_m)$ be the norm one torus of $K/k$ 
of dimension $n-1$ and $X$ be a smooth $k$-compactification of $T$. 
Let $G_v$ be the decomposition group of $G$ at a place $v$ of $k$.\\
{\rm (1)} If ${\rm Syl}_2(H)\not\simeq C_2,C_4,C_8$, then  
$A(T)\simeq\Sha(T)\simeq H^1(k,{\rm Pic}\,\overline{X})=0$ $($see Table $1$$)$.\\
{\rm (2)} If ${\rm Syl}_2(H)\simeq C_2,C_4,C_8$, then  
either 
{\rm (a)} $A(T)=0$ and $\Sha(T)\simeq H^1(k,{\rm Pic}\,\overline{X})\simeq\bZ/2\bZ$ 
or 
{\rm (b)} 
$A(T)\simeq H^1(k,{\rm Pic}\,\overline{X})\simeq\bZ/2\bZ$ and $\Sha(T)=0$ $($see Table $2$$)$,  
and the following conditions are equivalent:\\
 {\rm (b)} 
$A(T)\simeq H^1(k,{\rm Pic}\,\overline{X})\simeq\bZ/2\bZ$ and $\Sha(T)=0$;\\
{\rm (c)} there exists a place $v$ of $k$ such that 
\begin{align*}
\begin{cases}
V_4\leq G_v\ {\rm or}\ Q_8\leq G_v & {\rm if}\ \ {\rm Syl}_2(H)\simeq C_2,\\
D_4\leq G_v\ {\rm or}\ Q_8\leq G_v  & {\rm if}\ \ {\rm Syl}_2(H)\simeq C_4,\\
QD_8\leq G_v & {\rm if}\ \ {\rm Syl}_2(H)\simeq C_8.
\end{cases}
\end{align*}
\end{theorem}
\begin{theorem}\label{thmain2-2}
Let $k$ be a number field, %global field, %number field, 
$K/k$ be a field extension of degree $n$ 
and $L/k$ be the Galois closure of $K/k$. 
Assume that $G={\rm Gal}(L/k)\simeq J_1$ and $H={\rm Gal}(L/K)\lneq G$ 
where $J_1$ is the Janko group. 
Let $T=R^{(1)}_{K/k}(\bG_m)$ be the norm one torus of $K/k$ 
of dimension $n-1$ and $X$ be a smooth $k$-compactification of $T$. 
Let $G_v$ be the decomposition group of $G$ at a place $v$ of $k$.\\
{\rm (1)} If ${\rm Syl}_2(H)\not\simeq C_2$, then  
$A(T)\simeq\Sha(T)\simeq H^1(k,{\rm Pic}\,\overline{X})=0$ $($see Table $3$$)$.\\
{\rm (2)} If ${\rm Syl}_2(H)\simeq C_2$, then  
either 
{\rm (a)} $A(T)=0$ and $\Sha(T)\simeq H^1(k,{\rm Pic}\,\overline{X})\simeq\bZ/2\bZ$ 
or 
{\rm (b)} 
$A(T)\simeq H^1(k,{\rm Pic}\,\overline{X})\simeq\bZ/2\bZ$ and $\Sha(T)=0$ $($see Table $4$$)$,  
and the following conditions are equivalent:\\
 {\rm (b)} 
$A(T)\simeq H^1(k,{\rm Pic}\,\overline{X})\simeq\bZ/2\bZ$ and $\Sha(T)=0$;\\
{\rm (c)} there exists a place $v$ of $k$ such that $V_4\leq G_v$.
\end{theorem}
Note that 
a place $v$ of $k$ with non-cyclic decomposition group $G_v$ 
as in Theorem \ref{thmain2} (c) (resp. Theorem \ref{thmain2-2} (c)) 
must be ramified in $L$ because 
if $v$ is unramified, then $G_v$ is cyclic. 
%
%%%%%%%%%%%%%%%%%%%%%%%%%%%%%%%%%%%%%%%%%%%%%%%%
%
\newpage
\begin{center}
\vspace*{1mm}
Table $1$:  $H\lneq G\simeq M_{11}$ with $[G:H]=n$ and $H^1(k,{\rm Pic}\, \overline{X})\simeq H^1(G,[J_{G/H}]^{fl})=0$
\vspace*{2mm}\\
\begin{tabular}{llccc} 
$H$ & ${\rm Syl}_2(H)$ & $|H|$ & $n=[K:k]$ & $H^1(k,{\rm Pic}\, \overline{X})$ $\simeq H^1(G,[J_{G/H}]^{fl})$\\\hline
$\{1\}$ & $\{1\}$ & 1 & 7920 & 0\\
$C_3$ & $\{1\}$ & 3 & 2640 & 0\\
$V_4$ & $V_4$ & 4 & 1980 & 0\\
$C_5$ & $\{1\}$ & 5 & 1584 & 0\\
$Q_8$ & $Q_8$ & 8 & 990 & 0\\
$D_4$ & $D_4$ & 8 & 990 & 0\\
$C_3\times C_3$ & $\{1\}$ & 9 & 880 & 0\\
$C_{11}$ & $\{1\}$ & 11 & 720 & 0\\
$A_4$ & $V_4$ & 12 & 660 & 0\\
$D_6$ & $V_4$ & 12 & 660 & 0\\
$QD_8$ & $QD_8$ & 16 & 495 & 0\\
$\SL_2(\bF_3)$ & $Q_8$ & 24 & 330 & 0\\
$S_4$ & $D_4$ & 24 & 330 & 0\\
$S_3\times S_3$ & $V_4$ & 36 & 220 & 0\\
$\GL_2(\bF_3)$ & $QD_8$ & 48 & 165 & 0\\
$C_{11}\rtimes C_5$ & $\{1\}$ & 55 & 144 & 0\\
$A_5^{(1)}$ & $V_4$ & 60 & 132 & 0\\
$A_5^{(2)}$ & $V_4$ & 60 & 132 & 0\\
$(C_3\times C_3)\rtimes Q_8$ & $Q_8$ & 72 & 110 & 0\\
$(S_3\times S_3)\rtimes C_2$ & $D_4$ & 72 & 110 & 0\\
$S_5$ & $D_4$ & 120 & 66 & 0\\
$(C_3\times C_3)\rtimes QD_8$ & $QD_8$ & 144 & 55 & 0\\
$A_6$ & $D_4$ & 360 & 22 & 0\\
$\PSL_2(\bF_{11})$ & $V_4$ & 660 & 12 & 0\\
$M_{10}=A_6 . C_2$ & $QD_8$ & 720 & 11 & 0
\end{tabular}
\end{center}
%%%%%%%%%%%%%%%%%%%%%%%%%%%%%%%%%%%%%%%%%%%%%%%%
%
%\newpage
\begin{center}
\vspace*{5mm}
Table $2$:  $H\lneq G\simeq M_{11}$ with $[G:H]=n$ and $H^1(k,{\rm Pic}\, \overline{X})\simeq H^1(G,[J_{G/H}]^{fl})\simeq \bZ/2\bZ$
\vspace*{2mm}\\
\begin{tabular}{llccc} 
$H$ & ${\rm Syl}_2(H)$ & $|H|$ & $n=[K:k]$ & $H^1(k,{\rm Pic}\, \overline{X})$ $\simeq H^1(G,[J_{G/H}]^{fl})$\\\hline
$C_2$ & $C_2$ & 2 & 3960 & $\bZ/2\bZ$\\
$C_4$ & $C_4$ & 4 & 1980 & $\bZ/2\bZ$\\
$S_3^{(1)}$ & $C_2$ & 6 & 1320 & $\bZ/2\bZ$\\
$S_3^{(2)}$ & $C_2$ & 6 & 1320 & $\bZ/2\bZ$\\
$C_6$ & $C_2$ & 6 & 1320 & $\bZ/2\bZ$\\
$C_8$ & $C_8$ & 8 & 990 & $\bZ/2\bZ$\\
$D_5$ & $C_2$ & 10 & 792 & $\bZ/2\bZ$\\
$(C_3\times C_3)\rtimes C_2$ & $C_2$ & 18 & 440 & $\bZ/2\bZ$\\
$S_3\times C_3$ & $C_2$ & 18 & 440 & $\bZ/2\bZ$\\
$C_5\rtimes C_4$ & $C_4$ & 20 & 396 & $\bZ/2\bZ$\\
$((C_3\times C_3)\rtimes C_4)^{(1)}$ & $C_4$ & 36 & 220 & $\bZ/2\bZ$\\
$((C_3\times C_3)\rtimes C_4)^{(2)}$ & $C_4$ & 36 & 220 & $\bZ/2\bZ$\\
$(C_3\times C_3)\rtimes C_8$ & $C_8$ & 72 & 110 & $\bZ/2\bZ$
\end{tabular}
\end{center}~\vspace*{0mm}\\

%
%%%%%%%%%%%%%%%%%%%%%%%%%%%%%%%%%%%%%%%%%%%%%%%%
%
\newpage
\begin{center}
\vspace*{1mm}
Table $3$:  $H\lneq G\simeq J_1$ with $[G:H]=n$ and $H^1(k,{\rm Pic}\, \overline{X})\simeq H^1(G,[J_{G/H}]^{fl})=0$
\vspace*{2mm}\\
\begin{tabular}{llccc} 
$H$ & ${\rm Syl}_2(H)$ & $|H|$ & $n=[K:k]$ & $H^1(k,{\rm Pic}\, \overline{X})$ $\simeq H^1(G,[J_{G/H}]^{fl})$\\\hline
$\{1\}$ & $\{1\}$ & 1 & 175560 & 0\\
$C_3$ & $\{1\}$ & 3 & 58520 & 0\\
$V_4$ & $V_4$ & 4 & 43890 & 0\\
$C_5$ & $\{1\}$ & 5 & 35112 & 0\\
$C_7$ & $\{1\}$ & 7 & 25080 & 0\\
$C_2\times C_2\times C_2$ & $C_2\times C_2\times C_2$ & 8 & 21945 & 0\\
$C_{11}$ & $\{1\}$ & 11 & 15960 & 0\\
$A_4$ & $V_4$ & 12 & 14630 & 0\\
$D_6$ & $V_4$ & 12 & 14630 & 0\\
$C_{15}$ & $\{1\}$ & 15 & 11704 & 0\\
$C_{19}$ & $\{1\}$ & 19 & 9240 & 0\\
$D_{10}$ & $V_4$ & 20 & 8778 & 0\\
$C_7\rtimes C_3$ & $\{1\}$ & 21 & 8360 & 0\\
$A_4\times C_2$ & $C_2\times C_2\times C_2$ & 24 & 7315 & 0\\
$C_{11}\rtimes C_5$ & $\{1\}$ & 55 & 3192 & 0\\
$(C_2\times C_2\times C_2)\rtimes C_7$ & $C_2\times C_2\times C_2$ & 56 & 3135 & 0\\
$C_{19}\rtimes C_3$ & $\{1\}$ & 57 & 3080 & 0\\
$A_5^{(1)}$ & $V_4$ & 60 & 2926 & 0\\
$A_5^{(2)}$ & $V_4$ & 60 & 2926 & 0\\
$S_3\times D_5$ & $V_4$ & 60 & 2926 & 0\\
$A_5\times C_2$ & $C_2\times C_2\times C_2$ & 120 & 1463 & 0\\
$(C_2\times C_2\times C_2)\rtimes (C_7\rtimes C_3)$ & 
$C_2\times C_2\times C_2$ & 168 & 1045 & 0\\
$\PSL_2(\bF_{11})$ & $V_4$ & 660 & 266 & 0
\end{tabular}
\end{center}
%%%%%%%%%%%%%%%%%%%%%%%%%%%%%%%%%%%%%%%%%%%%%%%%
%
%\newpage
\begin{center}
\vspace*{5mm}
Table $4$:  $H\lneq G\simeq J_1$ with $[G:H]=n$ and $H^1(k,{\rm Pic}\, \overline{X})\simeq H^1(G,[J_{G/H}]^{fl})\simeq \bZ/2\bZ$
\vspace*{2mm}\\
\begin{tabular}{llccc} 
$H$ & ${\rm Syl}_2(H)$ & $|H|$ & $n=[K:k]$ & $H^1(k,{\rm Pic}\, \overline{X})$ $\simeq H^1(G,[J_{G/H}]^{fl})$\\\hline
$C_2$ & $C_2$ & 2 & 87780 & $\bZ/2\bZ$\\
$S_3^{(1)}$ & $C_2$ & 6 & 29260 & $\bZ/2\bZ$\\
$S_3^{(2)}$ & $C_2$ & 6 & 29260 & $\bZ/2\bZ$\\
$C_6$ & $C_2$ & 6 & 29260 & $\bZ/2\bZ$\\
$D_5^{(1)}$ & $C_2$ & 10 & 17556 & $\bZ/2\bZ$\\
$D_5^{(2)}$ & $C_2$ & 10 & 17556 & $\bZ/2\bZ$\\
$C_{10}$ & $C_2$ & 10 & 17556 & $\bZ/2\bZ$\\
$D_7$ & $C_2$ & 14 & 12540 & $\bZ/2\bZ$\\
$D_{11}$ & $C_2$ & 22 & 7980 & $\bZ/2\bZ$\\
$D_5\times C_3$ & $C_2$ & 30 & 5852 & $\bZ/2\bZ$\\
$D_{15}$ & $C_2$ & 30 & 5852 & $\bZ/2\bZ$\\
$S_3\times C_5$ & $C_2$ & 30 & 5852 & $\bZ/2\bZ$\\
$D_{19}$ & $C_2$ & 38 & 4620 & $\bZ/2\bZ$\\
$C_7\rtimes C_6$ & $C_2$ & 42 & 4180 & $\bZ/2\bZ$\\
$C_{11}\rtimes C_{10}$ & $C_2$ & 110 & 1596 & $\bZ/2\bZ$\\
$C_{19}\rtimes C_6$ & $C_2$ & 114 & 1540 & $\bZ/2\bZ$
\end{tabular}
\end{center}~\vspace*{0mm}\\
\newpage

By Ono's formula $\tau(T)=|H^1(k,\widehat{T})|/|\Sha(T)|$  
(see Ono \cite[Main theorem, page 68]{Ono63}), 
we get the Tamagawa number $\tau(T)$ of algebraic $k$-tori $T$ over a global field $k$ 
%with $G={\rm Gal}(L/k)\simeq M_{11}$ 
(see Voskresenskii \cite[Theorem 2, page 146]{Vos98}, 
Hoshi, Kanai and Yamasaki \cite[Section 8, Application 2]{HKY22}): 
\begin{corollary}
%Let $k$ be a global field, %number field, 
%$K/k$ be a field extension of degree $n$ 
%and $L/k$ be the Galois closure of $K/k$. 
%Assume that $G={\rm Gal}(L/k)\simeq M_{11}$  
%s the Mathieu group of degree $11$ and $H={\rm Gal}(L/K)\lneq G$.
%Let $T=R^{(1)}_{K/k}(\bG_m)$ be the norm one torus of $K/k$ 
%of dimension $n-1$ and $X$ be a smooth $k$-compactification of $T$. 
Let the notation be as in Theorem \ref{thmain2} 
(resp. Theorem \ref{thmain2-2}). 
Then the Tamagawa number 
$\tau(T)=1/|\Sha(T)|=1$ or $1/2$ where $\Sha(T)$ is given as in 
Theorem \ref{thmain2} (resp. Theorem \ref{thmain2-2}). 
\end{corollary}
\begin{proof}
By the definition, we have $0\to \bZ\xrightarrow{\varepsilon^\circ}\bZ[G/H]\to J_{G/H}\to 0$ where 
$J_{G/H}\simeq \widehat{T}={\rm Hom}(T,\bG_m)$. 
Then we get $H^1(G,\bZ[G/H])\to H^1(G,J_{G/H})$ 
$\xrightarrow{\delta}$ $H^2(G,\bZ)$ 
where $\delta$ is the connecting homomorphism. 
We have $H^2(G,\bZ)\simeq H^1(G,\bQ/\bZ)
={\rm Hom}(G,\bQ/\bZ)\simeq G^{ab}=G/[G,G]=1$ 
because $G\simeq M_{11}$ (resp. $J_1$) 
is a simple group 
and $H^1(G,\bZ[G/H])\simeq H^1(H,\bZ)={\rm Hom}(H,\bZ)=0$ by Shapiro's lemma. 
This implies that $H^1(G,J_{G/H})=0$. 
Hence the assertion follows from Ono's formula $\tau(T)=|H^1(k,\widehat{T})|/|\Sha(T)|$ 
and $H^1(k,\widehat{T})\simeq H^1(G,J_{G/H})=0$. 
\end{proof}
\begin{remark}
We also get the group of $R$-equivalence classes $T(k)/R\simeq H^1(G,[J_{G/H}]^{fl})=0$ or 
$\bZ/2\bZ$ as in Table $1$ and Table $2$ (resp. Table $3$ and Table $4$) 
where $k$ is a local field 
(see Colliot-Th\'{e}l\`{e}ne and Sansuc \cite[Corollary 5, page 201]{CTS77}, 
Voskresenskii \cite[Section 17.2]{Vos98} and 
Hoshi, Kanai and Yamasaki \cite[Section 7]{HKY22}). 
\end{remark}

%%%%%%%%%%%%%%
We organize this paper as follows. 
In Section \ref{S2}, we recall Drakokhrust and Platonov's method 
for the Hasse norm principle for $K/k$. 
In Section \ref{S3}, we give the proof of 
Theorem \ref{thmain1}, Theorem \ref{thmain1-2}, Theorem \ref{thmain2} 
and Theorem \ref{thmain2-2} 
using Drakokhrust and Platonov's method with 
the aid of GAP computations developed by 
Hoshi, Kanai and Yamasaki \cite{HKY22}, \cite{HKY23}. 
In Section \ref{S4} (resp. Section \ref{S5}), 
GAP computations which are used in the proof of 
Theorem \ref{thmain1} and Theorem \ref{thmain2} for $G\simeq M_{11}$
(resp. Theorem \ref{thmain1-2} and Theorem \ref{thmain2-2} for $G\simeq J_1$) 
are given.
Some related GAP algorithms are also available 
as in \cite{Norm1ToriHNP}.
%from\\ 
%\url{https://www.math.kyoto-u.ac.jp/~yamasaki/Algorithm/Norm1ToriHNP}.

%%%%%%%%%%%%%%%%%%%%%%%%%%%%%%%%%%%%%%%%%%%%%%%%
%\ssection{Preliminaries}

%
%%%%%%%%%%%%%%%%%%%%%%%%%%%%%%%%%
\section{Drakokhrust and Platonov's method}\label{S2}

Let $k$ be a number field, $K/k$ be a finite extension, 
$\bA_K^\times$ be the idele group of $K$ and 
$L/k$ be the Galois closure of $K/k$. 
% with ${\rm Gal}(L/k)\simeq$. 
Let $G={\rm Gal}(L/k)$ %be a transitive subgroup of $S_n$ 
and $H={\rm Gal}(L/K)$. %with $[G:H]=n$. 

For $x,y\in G$, we denote $[x,y]=x^{-1}y^{-1}xy$ the commutator of 
$x$ and $y$, and $[G,G]$ the commutator group of $G$. 
Let $V_k$ be the set of all places of $k$ 
and $G_v$ be the decomposition group of $G$ at $v\in V_k$. 

\begin{definition}[{Drakokhrust and Platonov \cite[page 350]{PD85a}, \cite[page 300]{DP87}}]
Let $k$ be a number field and 
$L\supset K\supset k$ be a tower of finite extensions 
where $L$ is normal over $k$. 
We call the group 
\begin{align*}
{\rm Obs}(K/k)=(N_{K/k}(\bA_K^\times)\cap k^\times)/N_{K/k}(K^\times)
\end{align*}
{\it the total obstruction to the Hasse norm principle for $K/k$} 
and 
\begin{align*}
{\rm Obs}_1(L/K/k)=\left(N_{K/k}(\bA_K^\times)\cap k^\times\right)/\left((N_{L/k}(\bA_L^\times)\cap k^\times)N_{K/k}(K^\times)\right)
\end{align*}
{\it the first obstruction to the Hasse norm principle for $K/k$ 
%with respect to 
corresponding to the tower 
$L\supset K\supset k$}. 
\end{definition}

Note that (i) 
${\rm Obs}(K/k)=1$ if and only if 
the Hasse norm principle holds for $K/k$; 
and (ii) ${\rm Obs}_1(L/K/k)
={\rm Obs}(K/k)/(N_{L/k}(\bA_L^\times)\cap k^\times)$. 

Drakokhrust and Platonov gave a formula 
for computing the first obstruction ${\rm Obs}_1(L/K/k)$: 
%for $K/k$: 

\begin{theorem}[{Drakokhrust and Platonov \cite[page 350]{PD85a}, \cite[pages 789--790]{PD85b}, \cite[Theorem 1]{DP87}}]\label{thDP2}
%Let $k$ be a number field, $K/k$ be a finite extension 
%and $L/k$ be the Galois closure of $K/k$.
Let $k$ be a number field, 
$L\supset K\supset k$ be a tower of finite extensions 
where 
$L$ is normal over $k$.  
Let $G={\rm Gal}(L/k)$ and $H={\rm Gal}(L/K)$. %with $[G:H]=n$. 
Then 
\begin{align*}
%\Sha(R^{(1)}_{K/k}(\bG_m))
{\rm Obs}_1(L/K/k)\simeq 
{\rm Ker}\, \psi_1/\varphi_1({\rm Ker}\, \psi_2)
\end{align*}
where 
\begin{align*}
\begin{CD}
H/[H,H] @>\psi_1 >> G/[G,G]\\
@AA\varphi_1 A @AA\varphi_2 A\\
\displaystyle{\bigoplus_{v\in V_k}\left(\bigoplus_{w\mid v} H_w/[H_w,H_w]\right)} @>\psi_2 >> 
\displaystyle{\bigoplus_{v\in V_k} G_v/[G_v,G_v]}, 
\end{CD}
\end{align*}
$\psi_1$, $\varphi_1$ and $\varphi_2$ are defined 
by the inclusions $H\subset G$, $H_w\subset H$ and $G_v\subset G$ respectively, and 
\begin{align*}
\psi_2(h[H_{w},H_{w}])=x^{-1}hx[G_v,G_v]
\end{align*}
for $h\in H_{w}=H\cap x^{-1}hx[G_v,G_v]$ $(x\in G)$.
\end{theorem}

%%%%%%%%%%%%%%%%%%%%%%%%%%%%%

Let $\psi_2^{v}$ be the restriction of $\psi_2$ to the subgroup 
$\bigoplus_{w\mid v} H_w/[H_w,H_w]$ with respect to $v\in V_k$ 
and $\psi_2^{\rm nr}$ (resp. $\psi_2^{\rm r}$) be 
the restriction of $\psi_2$ to the unramified (resp. the ramified) 
places $v$ of $k$. 
\begin{proposition}[{Drakokhrust and Platonov \cite{DP87}}]\label{propDP}
Let $k$, 
$L\supset K\supset k$, 
$G$ and $H$ be as in Theorem \ref{thDP2}.\\
{\rm (i)} $($\cite[Lemma 1]{DP87}$)$ 
Places $w_i\mid v$ of $K$ are in one-to-one correspondence 
with the set of double cosets in the decomposition 
$G=\cup_{i=1}^{r_v} Hx_iG_v$ where $H_{w_i}=H\cap x_iG_vx_i^{-1}$;\\
{\rm (ii)} $($\cite[Lemma 2]{DP87}$)$ 
If $G_{v_1}\leq G_{v_2}$, then $\varphi_1({\rm Ker}\,\psi_2^{v_1})\subset \varphi_1({\rm Ker}\,\psi_2^{v_2})$;\\
{\rm (iii)} $($\cite[Theorem 2]{DP87}$)$ 
$\varphi_1({\rm Ker}\,\psi_2^{\rm nr})=\Phi^G(H)/[H,H]$ 
where $\Phi^G(H)=\langle [h,x]\mid h\in H\cap xHx^{-1}, x\in G\rangle$;\\
{\rm (iv)} $($\cite[Lemma 8]{DP87}$)$ If $[K:k]=p^r$ $(r\geq 1)$ 
and ${\rm Obs}(K_p/k_p)=1$ where $k_p=L^{G_p}$, $K_p=L^{H_p}$, 
$G_p$ and $H_p\leq H\cap G_p$ are $p$-Sylow subgroups of $G$ and $H$ 
respectively, then ${\rm Obs}(K/k)=1$.
\end{proposition}

We note that if $L/k$ is an unramified extension, 
then $A(T)=0$ and $H^1(G,[J_{G/H}]^{fl})\simeq \Sha(T)\simeq 
{\rm Obs}(K/k)$ where $T=R^{(1)}_{K/k}(\bG_m)$ 
(see Theorem \ref{thV} and Theorem \ref{thOno}). 

\begin{theorem}[{Drakokhrust \cite[Theorem 1]{Dra89}, see also Opolka \cite[Satz 3]{Opo80}}]\label{thDra89}
Let $k$, 
$L\supset K\supset k$, 
$G$ and $H$ be as in Theorem \ref{thDP2}. 
Assume that $\widetilde{L}\supset L\supset k$ is 
a tower of Galois extensions with 
$\widetilde{G}={\rm Gal}(\widetilde{L}/k)$ 
and $\widetilde{H}={\rm Gal}(\widetilde{L}/K)$ 
which correspond to a central extension 
$1\to A\to \widetilde{G}\to G\to 1$ with 
$A\cap[\widetilde{G},\widetilde{G}]\simeq M(G)=H^2(G,\bC^\times)$; 
the Schur multiplier of $G$ 
$($this is equivalent to 
the inflation 
$M(G)\to M(\widetilde{G})$ being the zero map, 
see {\rm Beyl and Tappe \cite[Proposition 2.13, page 85]{BT82}}$)$. 
Then 
${\rm Obs}(K/k)={\rm Obs}_1(\widetilde{L}/K/k)$. 
In particular, if $\widetilde{G}$ is a Schur cover of $G$, 
i.e. $A\simeq M(G)$, then ${\rm Obs}(K/k)={\rm Obs}_1(\widetilde{L}/K/k)$. 
\end{theorem}

Indeed, Drakokhrust \cite[Theorem 1]{Dra89} shows that 
${\rm Obs}(K/k)\simeq 
{\rm Ker}\, \widetilde{\psi}_1/\widetilde{\varphi}_1({\rm Ker}\, \widetilde{\psi}_2)$ where the maps $\widetilde{\psi}_1, \widetilde{\psi}_2$ and $\widetilde{\varphi}_1$ are defined as in 
\cite[page 31, the paragraph before Proposition 1]{Dra89}. 
The proof of \cite[Proposition 1]{Dra89} shows that 
this group is the same as ${\rm Obs}_1(\widetilde{L}/K/k)$ 
(see also \cite[Lemma 2, Lemma 3 and Lemma 4]{Dra89}). 

Note that the Schur multiplier 
$M(G)\simeq H^2(G,\bC^\times)\simeq H^2(G,\bQ/\bZ)\simeq H^3(G,\bZ)$. 
Hence if $M(G)=0$, i.e. $\widetilde{L}=L$, then ${\rm Obs}(K/k)={\rm Obs}_1(L/K/k)$. 
In addition, if $L/k$ is unramified extension, then 
${\rm Obs}(K/k)={\rm Obs}_1(L/K/k)
={\rm Ker}\, \psi_1/\varphi_1({\rm Ker}\, \psi_2^{\rm nr})\simeq 
{\rm Ker}\, \psi_1/(\Phi^G(H)/[H,H])$ 
(see Proposition \ref{propDP} (iii)). 

%%%%%%%%%%%%%%%%%%%%%%%%%%%%%%%%%%%%%%%%%%%%%%%%%%%%%%%%%%%%%%%%%%%%%%
Hoshi, Kanai and Yamasaki \cite[Section 6]{HKY22} and \cite[Section 6]{HKY23}
made some related functions of GAP (\cite{GAP}) in order to perform 
Drakokhrust and Platonov's method (e.g. Theorem \ref{thDP2} and Theorem \ref{thDra89}) 
%for some cases 
which are also available as in \cite{Norm1ToriHNP}.
%from 
%\url{https://www.math.kyoto-u.ac.jp/~yamasaki/Algorithm/Norm1ToriHNP}. 
We will use such GAP functions in the proof of 
Theorem \ref{thmain1}, Theorem \ref{thmain1-2}, Theorem \ref{thmain2} 
and Theorem \ref{thmain2-2}. 

%%%%%%%%%%%%%%%%%%%%%%%%%%%%%%%%%%%%%%%%%%%%%%%%
\section{Proof of Theorem {\ref{thmain1}}, Theorem {\ref{thmain1-2}}, Theorem {\ref{thmain2}} and Theorem {\ref{thmain2-2}}}\label{S3}

Let $M_{11}$ be the Mathieu group of degree $11$ 
and $J_1$ be the Janko group of order $175560$. 
Let $K/k$ be a separable field extension and $L/k$ be the Galois closure 
of $K/k$. 
Assume that $G={\rm Gal}(L/k)\simeq M_{11}$ or $J_1$, 
and $H={\rm Gal}(L/K)\lneq G$.\\

{\it Proof of Theorem \ref{thmain1} and Theorem \ref{thmain1-2}.}

Assume that $G\simeq M_{11}$ (resp. $J_1$). 
Then we have the trivial Schur multiplier 
$M(G)\simeq H^2(G,\bC^\times)\simeq H^2(G,\bQ/\bZ)\simeq H^3(G,\bZ)=0$. 
By Theorem \ref{thDra89}, we have ${\rm Obs}(K/k)={\rm Obs}_1(L/K/k)$. 
It follows from Theorem \ref{thV} and Theorem \ref{thOno} that 
if $L/k$ is an unramified extension, 
then $A(T)=0$ and $H^1(G,[J_{G/H}]^{fl})\simeq \Sha(T)\simeq 
{\rm Obs}(K/k)$ where $T=R^{(1)}_{K/k}(\bG_m)$. 
This implies that 
${\rm Obs}(K/k)={\rm Obs}_1(L/K/k)
={\rm Ker}\, \psi_1/\varphi_1({\rm Ker}\, \psi_2^{\rm nr})$ 
when $L/k$ is an unramified extension. 

By applying the GAP functions  
{\tt FirstObstructionN($G$)} and {\tt FirstObstructionDnr($G$)} 
(see \cite[Section 6]{HKY22} and \cite[Section 6]{HKY23}), 
we obtain that $H^1(k,{\rm Pic}\,\overline{X})\simeq H^1(G,[J_{G/H}]^{fl})\simeq \Sha(T)\simeq $
${\rm Ker}\, \psi_1/\varphi_1({\rm Ker}\, \psi_2^{\rm nr})$ 
as in Table $1$ and Table $2$ (resp. Table $3$ and Table $4$), 
see Section \ref{S4} (resp. Section \ref{S5}) for GAP computations. 
For retract $k$-rationality, see also \cite[Section 3]{HKY22} and \cite[Section 4]{HKY23}.\qed\\

%%%%%%%%%%%%%%%%%%%%%%%%%%%%%%%
{\it Proof of Theorem \ref{thmain2} and Theorem \ref{thmain2-2}.}

Assume that $G\simeq M_{11}$ (resp. $J_1$). 
It follows from $M(G)=0$ and Theorem \ref{thDra89} that 
${\rm Obs}(K/k)={\rm Obs}_1(L/K/k)$. 
As the same in Hoshi, Kanai and Yamasaki \cite[Section 7, Proof of Theorem 1.3]{HKY23}, 
we apply the function {\tt FirstObstructionDr($G,G_{v_{r,s}}$)} 
to representatives of 
%%%%%%%%%%%%%%%%%
the orbit 
${\rm Orb}_{N_G(H)\backslash G/N_G(G_{v_{r,s}})}(G_{v_{r,s}})$ 
of $G_{v_{r,s}}\leq G$ under the conjugate action of $G$ 
which corresponds to the double coset 
$N_G(H)\backslash G/N_G(G_{v_{r,s}})$ 
with ${\rm Orb}_{G/N_G(G_{v_r})}(G_{v_r})$
$=$$\bigcup_{s=1}^{u_r}{\rm Orb}_{N_G(H)\backslash G/N_G(G_{v_{r,s}})}(G_{v_{r,s}})$ 
corresponding to $r$-th subgroup $G_{v_r}\leq G$ up to conjugacy 
%%%%%%%%%%%%%%%%%
via the function
\begin{center}
{\tt ConjugacyClassesSubgroupsNGHOrbitRep(ConjugacyClassesSubgroups($G$),$H$)}.
\end{center}
Then we can get the minimal elements of the $G_{v_{r,s}}$'s 
with ${\rm Ker}\, \psi_1/\varphi_1({\rm Ker}\, \psi_2)=0$ 
via the function 
\begin{center}
{\tt MinConjugacyClassesSubgroups($l$)}.
\end{center}

Finally, we get a necessary and sufficient condition for 
${\rm Obs}(K/k)={\rm Obs}_1(L/K/k)$ $=$ $1$ for each $H\lneq G$ 
in Table $2$ (resp. Table $4$) by the case-by-case analysis.  
%Alternatively, 
We can apply Macedo and Newton \cite[Corollary 3.4]{MN22} 
and then it is enough to check only the cases of 
$2$-Sylow subgroups of $H={\rm Gal}(L/K)$ in Table $2$ (resp. Table $4$), 
i.e. $H\simeq C_2, C_4, C_8$ (resp. $C_2$), 
see Section \ref{S4} (resp. Section \ref{S5}) for GAP computations. 
For $G\simeq J_1$, 
by applying Macedo and Newton \cite[Corollary 3.4]{MN22} again, 
we should check only one of the $H$'s in Table $4$ 
(e.g. $H\simeq C_{19}\rtimes C_6$ with the minimal 
$[G:H]=[K:k]=2^2\cdot 5\cdot 7\cdot 11=1540)$ 
because all the caseses have a 2-Sylow subgroup $\simeq C_2$ 
(see Section \ref{S5}).\qed
%
%%%%%%%%%%%%%%%%%%%%%%%%%%%%%%%%%%%%%%%%%%%%%%%%
\section{GAP computations: The $M_{11}$ case}\label{S4}
\begin{verbatim}
gap> Read("FlabbyResolutionFromBase.gap");
gap> Read("HNP.gap");

gap> M11:=MathieuGroup(11); # G=M11
Group([ (1,2,3,4,5,6,7,8,9,10,11), (3,7,11,8)(4,10,5,6) ])
gap> Order(M11); # |G|=7920=2^4*3^2*5*11
7920
gap> M11cs:=ConjugacyClassesSubgroups2(M11);; # subgroups H of G up to conjugacy
gap> Length(M11cs); # the number of H<=G up to conjugacy 
39
gap> M11H:=List([1..Length(M11cs)-1],x->Representative(M11cs[x]));; # H<G up to conjugacy 
gap> Length(M11H);
38
gap> List(M11H,Order);
[ 1, 2, 3, 4, 4, 5, 6, 6, 6, 8, 8, 8, 9, 10, 11, 12, 12, 16, 18, 18, 20, 24, 
  24, 36, 36, 36, 48, 55, 60, 60, 72, 72, 72, 120, 144, 360, 660, 720 ]
gap> List(M11H,StructureDescription);
[ "1", "C2", "C3", "C2 x C2", "C4", "C5", "S3", "S3", "C6", "Q8", "C8", "D8", 
  "C3 x C3", "D10", "C11", "A4", "D12", "QD16", "(C3 x C3) : C2", "C3 x S3", 
  "C5 : C4", "SL(2,3)", "S4", "(C3 x C3) : C4", "(C3 x C3) : C4", "S3 x S3", 
  "GL(2,3)", "C11 : C5", "A5", "A5", "(C3 x C3) : Q8", "(S3 x S3) : C2", 
  "(C3 x C3) : C8", "S5", "(C3 x C3) : QD16", "A6", "PSL(2,11)", "A6 . C2" ]
gap> GroupCohomology(M11,3); # the Schur multiplier H^3(M11,Z)=0 
[  ]
gap> Nker:=List(M11H,x->FirstObstructionN(M11,x).ker); # Obs1N
[ [ [  ], [ [  ], [  ] ] ], 
  [ [ 2 ], [ [ 2 ], [ [ 1 ] ] ] ], 
  [ [ 3 ], [ [ 3 ], [ [ 1 ] ] ] ], 
  [ [ 2, 2 ], [ [ 2, 2 ], 
  [ [ 1, 0 ], [ 0, 1 ] ] ] ], 
  [ [ 4 ], [ [ 4 ], [ [ 1 ] ] ] ], 
  [ [ 5 ], [ [ 5 ], [ [ 1 ] ] ] ], 
  [ [ 2 ], [ [ 2 ], [ [ 1 ] ] ] ], 
  [ [ 2 ], [ [ 2 ], [ [ 1 ] ] ] ], 
  [ [ 6 ], [ [ 6 ], [ [ 1 ] ] ] ], 
  [ [ 2, 2 ], [ [ 2, 2 ], 
  [ [ 1, 0 ], [ 0, 1 ] ] ] ], 
  [ [ 8 ], [ [ 8 ], [ [ 1 ] ] ] ], 
  [ [ 2, 2 ], [ [ 2, 2 ], 
  [ [ 1, 0 ], [ 0, 1 ] ] ] ], 
  [ [ 3, 3 ], [ [ 3, 3 ], 
  [ [ 1, 0 ], [ 0, 1 ] ] ] ], 
  [ [ 2 ], [ [ 2 ], [ [ 1 ] ] ] ], 
  [ [ 11 ], [ [ 11 ], [ [ 1 ] ] ] ], 
  [ [ 3 ], [ [ 3 ], [ [ 1 ] ] ] ], 
  [ [ 2, 2 ], [ [ 2, 2 ], 
  [ [ 1, 0 ], [ 0, 1 ] ] ] ], 
  [ [ 2, 2 ], [ [ 2, 2 ], 
  [ [ 1, 0 ], [ 0, 1 ] ] ] ], 
  [ [ 2 ], [ [ 2 ], [ [ 1 ] ] ] ], 
  [ [ 6 ], [ [ 6 ], [ [ 1 ] ] ] ], 
  [ [ 4 ], [ [ 4 ], [ [ 1 ] ] ] ], 
  [ [ 3 ], [ [ 3 ], [ [ 1 ] ] ] ], 
  [ [ 2 ], [ [ 2 ], [ [ 1 ] ] ] ], 
  [ [ 4 ], [ [ 4 ], [ [ 1 ] ] ] ], 
  [ [ 4 ], [ [ 4 ], [ [ 1 ] ] ] ], 
  [ [ 2, 2 ], [ [ 2, 2 ], 
  [ [ 1, 0 ], [ 0, 1 ] ] ] ], 
  [ [ 2 ], [ [ 2 ], [ [ 1 ] ] ] ], 
  [ [ 5 ], [ [ 5 ], [ [ 1 ] ] ] ], 
  [ [  ], [ [  ], [  ] ] ], 
  [ [  ], [ [  ], [  ] ] ], 
  [ [ 2, 2 ], [ [ 2, 2 ], 
  [ [ 1, 0 ], [ 0, 1 ] ] ] ], 
  [ [ 2, 2 ], [ [ 2, 2 ], 
  [ [ 1, 0 ], [ 0, 1 ] ] ] ], 
  [ [ 8 ], [ [ 8 ], [ [ 1 ] ] ] ], 
  [ [ 2 ], [ [ 2 ], [ [ 1 ] ] ] ], 
  [ [ 2, 2 ], [ [ 2, 2 ], 
  [ [ 1, 0 ], [ 0, 1 ] ] ] ], 
  [ [  ], [ [  ], [  ] ] ], 
  [ [  ], [ [  ], [  ] ] ], 
  [ [ 2 ], [ [ 2 ], [ [ 1 ] ] ] ] ]
gap> Dnr:=List(M11H,x->FirstObstructionDnr(M11,x).Dnr); # Obs1Dnr
[ [ [  ], [ [  ], [  ] ] ], 
  [ [  ], [ [ 2 ], [  ] ] ], 
  [ [ 3 ], [ [ 3 ], [ [ 1 ] ] ] ], 
  [ [ 2, 2 ], [ [ 2, 2 ], 
  [ [ 1, 0 ], [ 0, 1 ] ] ] ], 
  [ [ 2 ], [ [ 4 ], [ [ 2 ] ] ] ], 
  [ [ 5 ], [ [ 5 ], [ [ 1 ] ] ] ], 
  [ [  ], [ [ 2 ], [  ] ] ], 
  [ [  ], [ [ 2 ], [  ] ] ], 
  [ [ 3 ], [ [ 6 ], [ [ 2 ] ] ] ], 
  [ [ 2, 2 ], [ [ 2, 2 ], 
  [ [ 1, 0 ], [ 0, 1 ] ] ] ], 
  [ [ 4 ], [ [ 8 ], [ [ 2 ] ] ] ], 
  [ [ 2, 2 ], [ [ 2, 2 ], 
  [ [ 1, 0 ], [ 0, 1 ] ] ] ], 
  [ [ 3, 3 ], [ [ 3, 3 ], 
  [ [ 1, 0 ], [ 0, 1 ] ] ] ], 
  [ [  ], [ [ 2 ], [  ] ] ], 
  [ [ 11 ], [ [ 11 ], [ [ 1 ] ] ] ], 
  [ [ 3 ], [ [ 3 ], [ [ 1 ] ] ] ], 
  [ [ 2, 2 ], [ [ 2, 2 ], 
  [ [ 1, 0 ], [ 0, 1 ] ] ] ], 
  [ [ 2, 2 ], [ [ 2, 2 ], 
  [ [ 1, 0 ], [ 0, 1 ] ] ] ], 
  [ [  ], [ [ 2 ], [  ] ] ], 
  [ [ 3 ], [ [ 6 ], [ [ 2 ] ] ] ], 
  [ [ 2 ], [ [ 4 ], [ [ 2 ] ] ] ], 
  [ [ 3 ], [ [ 3 ], [ [ 1 ] ] ] ], 
  [ [ 2 ], [ [ 2 ], [ [ 1 ] ] ] ], 
  [ [ 2 ], [ [ 4 ], [ [ 2 ] ] ] ], 
  [ [ 2 ], [ [ 4 ], [ [ 2 ] ] ] ], 
  [ [ 2, 2 ], [ [ 2, 2 ], 
  [ [ 1, 0 ], [ 0, 1 ] ] ] ], 
  [ [ 2 ], [ [ 2 ], [ [ 1 ] ] ] ], 
  [ [ 5 ], [ [ 5 ], [ [ 1 ] ] ] ], 
  [ [  ], [ [  ], [  ] ] ], 
  [ [  ], [ [  ], [  ] ] ], 
  [ [ 2, 2 ], [ [ 2, 2 ], [ [ 1, 0 ], [ 0, 1 ] ] ] ], 
  [ [ 2, 2 ], [ [ 2, 2 ], [ [ 1, 0 ], [ 0, 1 ] ] ] ], 
  [ [ 4 ], [ [ 8 ], [ [ 2 ] ] ] ], 
  [ [ 2 ], [ [ 2 ], [ [ 1 ] ] ] ], 
  [ [ 2, 2 ], [ [ 2, 2 ], [ [ 1, 0 ], [ 0, 1 ] ] ] ], 
  [ [  ], [ [  ], [  ] ] ], 
  [ [  ], [ [  ], [  ] ] ], 
  [ [ 2 ], [ [ 2 ], [ [ 1 ] ] ] ] ]
gap> H1F:=List([1..Length(M11H)],x->AbelianInvariantsGoverH(Nker[x][2],Dnr[x][2])); 
# abelian invariants of Nker/Dnr=Obs1N/Obs1Dnr=H^1(G,F) with F=[J_{G/H}]^{fl} 
[ [  ], [ 2 ], [  ], [  ], [ 2 ], [  ], [ 2 ], [ 2 ], [ 2 ], [  ], [ 2 ], 
  [  ], [  ], [ 2 ], [  ], [  ], [  ], [  ], [ 2 ], [ 2 ], [ 2 ], [  ], [  ], 
  [ 2 ], [ 2 ], [  ], [  ], [  ], [  ], [  ], [  ], [  ], [ 2 ], [  ], [  ], 
  [  ], [  ], [  ] ]
gap> Collected(H1F);
[ [ [  ], 25 ], [ [ 2 ], 13 ] ]
gap> H1F1:=Filtered([1..Length(M11H)],x->H1F[x]=[]); # H^1(G,F)=1
[ 1, 3, 4, 6, 10, 12, 13, 15, 16, 17, 18, 22, 23, 26, 27, 28, 29, 30, 31, 32, 
  34, 35, 36, 37, 38 ]
gap> Length(H1F1);
25
gap> List(M11H{H1F1},StructureDescription);
[ "1", "C3", "C2 x C2", "C5", "Q8", "D8", "C3 x C3", "C11", "A4", "D12", "QD16",
  "SL(2,3)", "S4", "S3 x S3", "GL(2,3)", "C11 : C5", "A5", "A5", "(C3 x C3) : Q8",
  "(S3 x S3) : C2", "S5", "(C3 x C3) : QD16", "A6", "PSL(2,11)", "A6 . C2" ]
gap> H1F2:=Filtered([1..Length(M11H)],x->H1F[x]=[2]); # H^1(G,F)=Z/2Z
[ 2, 5, 7, 8, 9, 11, 14, 19, 20, 21, 24, 25, 33 ]
gap> Length(H1F2);
13
gap> List(M11H{H1F2},StructureDescription);
[ "C2", "C4", "S3", "S3", "C6", "C8", "D10", "(C3 x C3) : C2", "C3 x S3", "C5 : C4",
  "(C3 x C3) : C4", "(C3 x C3) : C4", "(C3 x C3) : C8" ]

gap> c1:=Filtered(M11H,x->Order(x) mod 2=0);;
gap> Length(c1);
32
gap> c2:=Filtered(c1,x->IsCyclic(SylowSubgroup(x,2)));;
gap> Length(c2);
13
gap> List(c2,x->Position(M11H,x));
[ 2, 5, 7, 8, 9, 11, 14, 19, 20, 21, 24, 25, 33 ]
gap> last=H1F2;
true
\end{verbatim}~\vspace*{-5mm}\\

\begin{verbatim}
gap> FirstObstructionN(M11,M11H[2]).ker; # H=C2
[ [ 2 ], [ [ 2 ], [ [ 1 ] ] ] ]
gap> FirstObstructionDnr(M11,M11H[2]).Dnr;
[ [  ], [ [ 2 ], [  ] ] ]
gap> HNPtruefalsefn:=x->FirstObstructionDr(M11,x,M11H[2]).Dr[1]=[2];
function( x ) ... end
gap> GcsH:=ConjugacyClassesSubgroupsNGHOrbitRep(M11cs,M11H[2]);;
gap> GcsHHNPtf:=List(GcsH,x->List(x,HNPtruefalsefn));;
gap> Collected(List(GcsHHNPtf,Set));
[ [ [ true ], 20 ], [ [ false ], 19 ] ]
gap> GcsHNPfalse:=List(Filtered([1..Length(M11cs)],
> x->false in GcsHHNPtf[x]),y->M11cs[y]);;
gap> Length(GcsHNPfalse);
19
gap> GcsHNPtrue:=List(Filtered([1..Length(M11cs)],
> x->true in GcsHHNPtf[x]),y->M11cs[y]);;
gap> Length(GcsHNPtrue);
20
gap> Collected(List(GcsHNPfalse,x->StructureDescription(Representative(x))));
[ [ "(C3 x C3) : C2", 1 ], [ "(C3 x C3) : C4", 2 ], [ "(C3 x C3) : C8", 1 ], 
  [ "1", 1 ], [ "C11", 1 ], [ "C11 : C5", 1 ], [ "C2", 1 ], [ "C3", 1 ], 
  [ "C3 x C3", 1 ], [ "C3 x S3", 1 ], [ "C4", 1 ], [ "C5", 1 ], 
  [ "C5 : C4", 1 ], [ "C6", 1 ], [ "C8", 1 ], [ "D10", 1 ], [ "S3", 2 ] ]
gap> Collected(List(GcsHNPtrue,x->StructureDescription(Representative(x))));
[ [ "(C3 x C3) : Q8", 1 ], [ "(C3 x C3) : QD16", 1 ], [ "(S3 x S3) : C2", 1 ],
  [ "A4", 1 ], [ "A5", 2 ], [ "A6", 1 ], [ "A6 . C2", 1 ], [ "C2 x C2", 1 ], 
  [ "D12", 1 ], [ "D8", 1 ], [ "GL(2,3)", 1 ], [ "M11", 1 ], 
  [ "PSL(2,11)", 1 ], [ "Q8", 1 ], [ "QD16", 1 ], [ "S3 x S3", 1 ], 
  [ "S4", 1 ], [ "S5", 1 ], [ "SL(2,3)", 1 ] ]
gap> GcsHNPtrueMin:=MinConjugacyClassesSubgroups(GcsHNPtrue);;
gap> Collected(List(GcsHNPtrueMin,x->StructureDescription(Representative(x))));
[ [ "C2 x C2", 1 ], [ "Q8", 1 ] ]
\end{verbatim}~\vspace*{-5mm}\\

\begin{verbatim}
gap> FirstObstructionN(M11,M11H[5]).ker; # H=C4
[ [ 4 ], [ [ 4 ], [ [ 1 ] ] ] ]
gap> FirstObstructionDnr(M11,M11H[5]).Dnr;
[ [ 2 ], [ [ 4 ], [ [ 2 ] ] ] ]
gap> HNPtruefalsefn:=x->FirstObstructionDr(M11,x,M11H[5]).Dr[1]=[4];
function( x ) ... end
gap> GcsH:=ConjugacyClassesSubgroupsNGHOrbitRep(M11cs,M11H[5]);;
gap> GcsHHNPtf:=List(GcsH,x->List(x,HNPtruefalsefn));;
gap> Collected(List(GcsHHNPtf,Set));
[ [ [ true ], 13 ], [ [ false ], 26 ] ]
gap> GcsHNPfalse:=List(Filtered([1..Length(M11cs)],
> x->false in GcsHHNPtf[x]),y->M11cs[y]);;
gap> Length(GcsHNPfalse);
26
gap> GcsHNPtrue:=List(Filtered([1..Length(M11cs)],
> x->true in GcsHHNPtf[x]),y->M11cs[y]);;
gap> Length(GcsHNPtrue);
13
gap> Collected(List(GcsHNPfalse,x->StructureDescription(Representative(x))));
[ [ "(C3 x C3) : C2", 1 ], [ "(C3 x C3) : C4", 2 ], [ "(C3 x C3) : C8", 1 ], 
  [ "1", 1 ], [ "A4", 1 ], [ "A5", 2 ], [ "C11", 1 ], [ "C11 : C5", 1 ], 
  [ "C2", 1 ], [ "C2 x C2", 1 ], [ "C3", 1 ], [ "C3 x C3", 1 ], 
  [ "C3 x S3", 1 ], [ "C4", 1 ], [ "C5", 1 ], [ "C5 : C4", 1 ], [ "C6", 1 ], 
  [ "C8", 1 ], [ "D10", 1 ], [ "D12", 1 ], [ "PSL(2,11)", 1 ], [ "S3", 2 ], 
  [ "S3 x S3", 1 ] ]
gap> Collected(List(GcsHNPtrue,x->StructureDescription(Representative(x))));
[ [ "(C3 x C3) : Q8", 1 ], [ "(C3 x C3) : QD16", 1 ], [ "(S3 x S3) : C2", 1 ],
  [ "A6", 1 ], [ "A6 . C2", 1 ], [ "D8", 1 ], [ "GL(2,3)", 1 ], [ "M11", 1 ], 
  [ "Q8", 1 ], [ "QD16", 1 ], [ "S4", 1 ], [ "S5", 1 ], [ "SL(2,3)", 1 ] ]
gap> GcsHNPtrueMin:=MinConjugacyClassesSubgroups(GcsHNPtrue);;
gap> Collected(List(GcsHNPtrueMin,x->StructureDescription(Representative(x))));
[ [ "D8", 1 ], [ "Q8", 1 ] ]
\end{verbatim}~\vspace*{-5mm}\\

\begin{verbatim}
gap> FirstObstructionN(M11,M11H[11]).ker; # H=C8
[ [ 8 ], [ [ 8 ], [ [ 1 ] ] ] ]
gap> FirstObstructionDnr(M11,M11H[11]).Dnr;
[ [ 4 ], [ [ 8 ], [ [ 2 ] ] ] ]
gap> HNPtruefalsefn:=x->FirstObstructionDr(M11,x,M11H[11]).Dr[1]=[8];
function( x ) ... end
gap> GcsH:=ConjugacyClassesSubgroupsNGHOrbitRep(M11cs,M11H[11]);;
gap> GcsHHNPtf:=List(GcsH,x->List(x,HNPtruefalsefn));;
gap> Collected(List(GcsHHNPtf,Set));
[ [ [ true ], 5 ], [ [ false ], 34 ] ]
gap> GcsHNPfalse:=List(Filtered([1..Length(M11cs)],
> x->false in GcsHHNPtf[x]),y->M11cs[y]);;
gap> Length(GcsHNPfalse);
34
gap> GcsHNPtrue:=List(Filtered([1..Length(M11cs)],
> x->true in GcsHHNPtf[x]),y->M11cs[y]);;
gap> Length(GcsHNPtrue);
5
gap> Collected(List(GcsHNPfalse,x->StructureDescription(Representative(x))));
[ [ "(C3 x C3) : C2", 1 ], [ "(C3 x C3) : C4", 2 ], [ "(C3 x C3) : C8", 1 ], 
  [ "(C3 x C3) : Q8", 1 ], [ "(S3 x S3) : C2", 1 ], [ "1", 1 ], [ "A4", 1 ], 
  [ "A5", 2 ], [ "A6", 1 ], [ "C11", 1 ], [ "C11 : C5", 1 ], [ "C2", 1 ], 
  [ "C2 x C2", 1 ], [ "C3", 1 ], [ "C3 x C3", 1 ], [ "C3 x S3", 1 ], 
  [ "C4", 1 ], [ "C5", 1 ], [ "C5 : C4", 1 ], [ "C6", 1 ], [ "C8", 1 ], 
  [ "D10", 1 ], [ "D12", 1 ], [ "D8", 1 ], [ "PSL(2,11)", 1 ], [ "Q8", 1 ], 
  [ "S3", 2 ], [ "S3 x S3", 1 ], [ "S4", 1 ], [ "S5", 1 ], [ "SL(2,3)", 1 ] ]
gap> Collected(List(GcsHNPtrue,x->StructureDescription(Representative(x))));
[ [ "(C3 x C3) : QD16", 1 ], [ "A6 . C2", 1 ], [ "GL(2,3)", 1 ], 
  [ "M11", 1 ], [ "QD16", 1 ] ]
gap> GcsHNPtrueMin:=MinConjugacyClassesSubgroups(GcsHNPtrue);;
gap> Collected(List(GcsHNPtrueMin,x->StructureDescription(Representative(x))));
[ [ "QD16", 1 ] ]
\end{verbatim}~\vspace*{-5mm}\\
\section{GAP computations: The $J_1$ case}\label{S5}
\begin{verbatim}
gap> Read("FlabbyResolutionFromBase.gap");
gap> Read("HNP.gap");

gap> J1:=SimpleGroup("J1");; # G=J1
gap> Order(J1); # |G|=175560=2^3*3*5*7*11*19
175560
gap> J1cs:=ConjugacyClassesSubgroups2(J1);; # subgroups H of G up to conjugacy
gap> Length(J1cs);; # the number of H<=G up to conjugacy 
40
gap> J1H:=List([1..Length(J1cs)-1],x->Representative(J1cs[x]));; # H<G up to conjugacy 
gap> Length(J1H);
39
gap> List(J1H,Order);
[ 1, 2, 3, 4, 5, 6, 6, 6, 7, 8, 10, 10, 10, 11, 12, 12, 14, 15, 19, 20, 21, 22,
  60, 60, 60, 110, 114, 120, 168, 660 ]
gap> List(J1H,Order);
[ 1, 2, 3, 4, 5, 6, 6, 6, 7, 8, 10, 10, 10, 11, 12, 12, 14, 15, 19, 20, 21,
  22, 24, 30, 30, 30, 38, 42, 55, 56, 57, 60, 60, 60, 110, 114, 120, 168, 660 ]
gap> List(J1H,StructureDescription);
[ "1", "C2", "C3", "C2 x C2", "C5", "S3", "S3", "C6", "C7", "C2 x C2 x C2",
  "D10", "D10", "C10", "C11", "A4", "D12", "D14", "C15", "C19", "D20",
  "C7 : C3", "D22", "C2 x A4", "C3 x D10", "D30", "C5 x S3", "D38",
  "C7 : C6", "C11 : C5", "(C2 x C2 x C2) : C7", "C19 : C3", "A5", "A5",
  "S3 x D10", "C11 : C10", "C19 : C6", "C2 x A5",
  "(C2 x C2 x C2) : (C7 : C3)", "PSL(2,11)" ]
gap> GroupCohomology(J1,3); # the Schur multiplier H^3(J1,Z)=0 
[  ]
gap> Nker:=List(J1H,x->FirstObstructionN(J1,x).ker); # Obs1N
[ [ [  ], [ [  ], [  ] ] ], 
  [ [ 2 ], [ [ 2 ], [ [ 1 ] ] ] ], 
  [ [ 3 ], [ [ 3 ], [ [ 1 ] ] ] ], 
  [ [ 2, 2 ], [ [ 2, 2 ], [ [ 1, 0 ], [ 0, 1 ] ] ] ], 
  [ [ 5 ], [ [ 5 ], [ [ 1 ] ] ] ], 
  [ [ 2 ], [ [ 2 ], [ [ 1 ] ] ] ], 
  [ [ 2 ], [ [ 2 ], [ [ 1 ] ] ] ], 
  [ [ 6 ], [ [ 6 ], [ [ 1 ] ] ] ], 
  [ [ 7 ], [ [ 7 ], [ [ 1 ] ] ] ], 
  [ [ 2, 2, 2 ], [ [ 2, 2, 2 ], [ [ 1, 0, 0 ], [ 0, 1, 0 ], [ 0, 0, 1 ] ] ] ],
  [ [ 2 ], [ [ 2 ], [ [ 1 ] ] ] ], 
  [ [ 2 ], [ [ 2 ], [ [ 1 ] ] ] ], 
  [ [ 10 ], [ [ 10 ], [ [ 1 ] ] ] ], 
  [ [ 11 ], [ [ 11 ], [ [ 1 ] ] ] ], 
  [ [ 3 ], [ [ 3 ], [ [ 1 ] ] ] ], 
  [ [ 2, 2 ], [ [ 2, 2 ], [ [ 1, 0 ], [ 0, 1 ] ] ] ], 
  [ [ 2 ], [ [ 2 ], [ [ 1 ] ] ] ], 
  [ [ 15 ], [ [ 15 ], [ [ 1 ] ] ] ], 
  [ [ 19 ], [ [ 19 ], [ [ 1 ] ] ] ], 
  [ [ 2, 2 ], [ [ 2, 2 ], [ [ 1, 0 ], [ 0, 1 ] ] ] ], 
  [ [ 3 ], [ [ 3 ], [ [ 1 ] ] ] ], 
  [ [ 2 ], [ [ 2 ], [ [ 1 ] ] ] ], 
  [ [ 6 ], [ [ 6 ], [ [ 1 ] ] ] ], 
  [ [ 6 ], [ [ 6 ], [ [ 1 ] ] ] ], 
  [ [ 2 ], [ [ 2 ], [ [ 1 ] ] ] ], 
  [ [ 10 ], [ [ 10 ], [ [ 1 ] ] ] ], 
  [ [ 2 ], [ [ 2 ], [ [ 1 ] ] ] ], 
  [ [ 6 ], [ [ 6 ], [ [ 1 ] ] ] ], 
  [ [ 5 ], [ [ 5 ], [ [ 1 ] ] ] ], 
  [ [ 7 ], [ [ 7 ], [ [ 1 ] ] ] ], 
  [ [ 3 ], [ [ 3 ], [ [ 1 ] ] ] ], 
  [ [  ], [ [  ], [  ] ] ], 
  [ [  ], [ [  ], [  ] ] ], 
  [ [ 2, 2 ], [ [ 2, 2 ], [ [ 1, 0 ], [ 0, 1 ] ] ] ], 
  [ [ 10 ], [ [ 10 ], [ [ 1 ] ] ] ], 
  [ [ 6 ], [ [ 6 ], [ [ 1 ] ] ] ], 
  [ [ 2 ], [ [ 2 ], [ [ 1 ] ] ] ], 
  [ [ 3 ], [ [ 3 ], [ [ 1 ] ] ] ], 
  [ [  ], [ [  ], [  ] ] ] ]
gap> Dnr:=List(J1H,x->FirstObstructionDnr(J1,x).Dnr); # Obs1Dnr
[ [ [  ], [ [  ], [  ] ] ], 
  [ [  ], [ [ 2 ], [  ] ] ], 
  [ [ 3 ], [ [ 3 ], [ [ 1 ] ] ] ], 
  [ [ 2, 2 ], [ [ 2, 2 ], 
  [ [ 1, 0 ], [ 0, 1 ] ] ] ], 
  [ [ 5 ], [ [ 5 ], [ [ 1 ] ] ] ], 
  [ [  ], [ [ 2 ], [  ] ] ], 
  [ [  ], [ [ 2 ], [  ] ] ], 
  [ [ 3 ], [ [ 6 ], [ [ 2 ] ] ] ], 
  [ [ 7 ], [ [ 7 ], [ [ 1 ] ] ] ], 
  [ [ 2, 2, 2 ], [ [ 2, 2, 2 ], 
  [ [ 1, 0, 0 ], [ 0, 1, 0 ], [ 0, 0, 1 ] ] ] ],
  [ [  ], [ [ 2 ], [  ] ] ], 
  [ [  ], [ [ 2 ], [  ] ] ], 
  [ [ 5 ], [ [ 10 ], [ [ 2 ] ] ] ], 
  [ [ 11 ], [ [ 11 ], [ [ 1 ] ] ] ], 
  [ [ 3 ], [ [ 3 ], [ [ 1 ] ] ] ], 
  [ [ 2, 2 ], [ [ 2, 2 ], [ [ 1, 0 ], [ 0, 1 ] ] ] ], 
  [ [  ], [ [ 2 ], [  ] ] ], 
  [ [ 15 ], [ [ 15 ], [ [ 1 ] ] ] ], 
  [ [ 19 ], [ [ 19 ], [ [ 1 ] ] ] ], 
  [ [ 2, 2 ], [ [ 2, 2 ], [ [ 1, 0 ], [ 0, 1 ] ] ] ], 
  [ [ 3 ], [ [ 3 ], [ [ 1 ] ] ] ], 
  [ [  ], [ [ 2 ], [  ] ] ], 
  [ [ 6 ], [ [ 6 ], [ [ 1 ] ] ] ], 
  [ [ 3 ], [ [ 6 ], [ [ 2 ] ] ] ], 
  [ [  ], [ [ 2 ], [  ] ] ], 
  [ [ 5 ], [ [ 10 ], [ [ 2 ] ] ] ], 
  [ [  ], [ [ 2 ], [  ] ] ], 
  [ [ 3 ], [ [ 6 ], [ [ 2 ] ] ] ], 
  [ [ 5 ], [ [ 5 ], [ [ 1 ] ] ] ], 
  [ [ 7 ], [ [ 7 ], [ [ 1 ] ] ] ], 
  [ [ 3 ], [ [ 3 ], [ [ 1 ] ] ] ], 
  [ [  ], [ [  ], [  ] ] ], 
  [ [  ], [ [  ], [  ] ] ], 
  [ [ 2, 2 ], [ [ 2, 2 ], [ [ 1, 0 ], [ 0, 1 ] ] ] ], 
  [ [ 5 ], [ [ 10 ], [ [ 2 ] ] ] ], 
  [ [ 3 ], [ [ 6 ], [ [ 2 ] ] ] ], 
  [ [ 2 ], [ [ 2 ], [ [ 1 ] ] ] ], 
  [ [ 3 ], [ [ 3 ], [ [ 1 ] ] ] ], 
  [ [  ], [ [  ], [  ] ] ] ]
gap> H1F:=List([1..Length(J1H)],x->AbelianInvariantsGoverH(Nker[x][2],Dnr[x][2]));
# abelian invariants of Nker/Dnr=Obs1N/Obs1Dnr=H^1(G,F) with F=[J_{G/H}]^{fl} 
[ [  ], [ 2 ], [  ], [  ], [  ], [ 2 ], [ 2 ], [ 2 ], [  ], [  ], [ 2 ], 
  [ 2 ], [ 2 ], [  ], [  ], [  ], [ 2 ], [  ], [  ], [  ], [  ], [ 2 ], [  ], 
  [ 2 ], [ 2 ], [ 2 ], [ 2 ], [ 2 ], [  ], [  ], [  ], [  ], [  ], [  ], 
  [ 2 ], [ 2 ], [  ], [  ], [  ] ]
gap> Collected(H1F);
[ [ [  ], 23 ], [ [ 2 ], 16 ] ]
gap> H1F1:=Filtered([1..Length(J1H)],x->H1F[x]=[]); # H^1(G,F)=1
[ 1, 3, 4, 5, 9, 10, 14, 15, 16, 18, 19, 20, 21, 23, 29, 30, 31, 32, 33, 34, 
  37, 38, 39 ]
gap> Length(H1F1);
23
gap> List(J1H{H1F1},StructureDescription);
[ "1", "C3", "C2 x C2", "C5", "C7", "C2 x C2 x C2", "C11", "A4", "D12",
  "C15", "C19", "D20", "C7 : C3", "C2 x A4", "C11 : C5",
  "(C2 x C2 x C2) : C7", "C19 : C3", "A5", "A5", "S3 x D10", "C2 x A5",
  "(C2 x C2 x C2) : (C7 : C3)", "PSL(2,11)" ]
gap> H1F2:=Filtered([1..Length(J1H)],x->H1F[x]=[2]); # H^1(G,F)=Z/2Z
[ 2, 6, 7, 8, 11, 12, 13, 17, 22, 24, 25, 26, 27, 28, 35, 36 ]
gap> Length(H1F2);
16
gap> List(J1H{H1F2},StructureDescription);
[ "C2", "S3", "S3", "C6", "D10", "D10", "C10", "D14", "D22", "C3 x D10",
  "D30", "C5 x S3", "D38", "C7 : C6", "C11 : C10", "C19 : C6" ]

gap> Collected(List(J1H{H1F1},x->StructureDescription(SylowSubgroup(x,2))));
[ [ "1", 10 ], [ "C2 x C2", 8 ], [ "C2 x C2 x C2", 5 ] ]
gap> Collected(List(J1H{H1F2},x->StructureDescription(SylowSubgroup(x,2))));
[ [ "C2", 16 ] ]
\end{verbatim}~\vspace*{-5mm}\\

\begin{verbatim}
gap> FirstObstructionN(J1,J1H[36]).ker; # H=C19:C6
[ [ 6 ], [ [ 6 ], [ [ 1 ] ] ] ]
gap> FirstObstructionDnr(J1,J1H[36]).Dnr;
[ [ 3 ], [ [ 6 ], [ [ 2 ] ] ] ]
gap> HNPtruefalsefn:=x->FirstObstructionDr(J1,x,J1H[36]).Dr[1] in [[2],[6]];
function( x ) ... end
gap> GcsH:=ConjugacyClassesSubgroupsNGHOrbitRep(J1cs,J1H[36]);;
gap> GcsHHNPtf:=List(GcsH,x->List(x,HNPtruefalsefn));;
gap> Collected(List(GcsHHNPtf,Set));
[ [ [ true ], 14 ], [ [ false ], 26 ] ]
gap> GcsHNPfalse:=List(Filtered([1..Length(J1cs)],
> x->false in GcsHHNPtf[x]),y->J1cs[y]);;
gap> Length(GcsHNPfalse);
26
gap> GcsHNPtrue:=List(Filtered([1..Length(J1cs)],
> x->true in GcsHHNPtf[x]),y->J1cs[y]);;
gap> Length(GcsHNPtrue);
14
gap> Collected(List(GcsHNPfalse,x->StructureDescription(Representative(x))));
[ [ "1", 1 ], [ "C10", 1 ], [ "C11", 1 ], [ "C11 : C10", 1 ], 
  [ "C11 : C5", 1 ], [ "C15", 1 ], [ "C19", 1 ], [ "C19 : C3", 1 ], 
  [ "C19 : C6", 1 ], [ "C2", 1 ], [ "C3", 1 ], [ "C3 x D10", 1 ], 
  [ "C5", 1 ], [ "C5 x S3", 1 ], [ "C6", 1 ], [ "C7", 1 ], [ "C7 : C3", 1 ], 
  [ "C7 : C6", 1 ], [ "D10", 2 ], [ "D14", 1 ], [ "D22", 1 ], [ "D30", 1 ], 
  [ "D38", 1 ], [ "S3", 2 ] ]
gap> Collected(List(GcsHNPtrue,x->StructureDescription(Representative(x))));
[ [ "(C2 x C2 x C2) : (C7 : C3)", 1 ], [ "(C2 x C2 x C2) : C7", 1 ], 
  [ "A4", 1 ], [ "A5", 2 ], [ "C2 x A4", 1 ], [ "C2 x A5", 1 ], 
  [ "C2 x C2", 1 ], [ "C2 x C2 x C2", 1 ], [ "D12", 1 ], [ "D20", 1 ], 
  [ "J1", 1 ], [ "PSL(2,11)", 1 ], [ "S3 x D10", 1 ] ]
gap> GcsHNPtrueMin:=MinConjugacyClassesSubgroups(GcsHNPtrue);;
gap> Collected(List(GcsHNPtrueMin,x->StructureDescription(Representative(x))));
[ [ "C2 x C2", 1 ] ]
\end{verbatim}
\newpage

\end{document}